\documentclass{article}%
\usepackage{graphicx}
\usepackage{amsmath,amssymb,amsfonts}%
\usepackage{theorem}%
\usepackage{color}%
\usepackage{hyperref}
\usepackage[font={small, it}]{caption}
\usepackage{caption}
\usepackage{subcaption}
\usepackage{tikz}

\theoremstyle{change}%
\newtheorem{definition}{Definition:}[section]%
\newtheorem{proposition}[definition]{Proposition:}%
\newtheorem{theorem}[definition]{Theorem:}%
\newtheorem{lemma}[definition]{Lemma:}%
\newtheorem{corollary}[definition]{Corollary:}%
{\theorembodyfont{\rmfamily}\newtheorem{remark}[definition]{Remark:}}%
{\theorembodyfont{\rmfamily}\newtheorem{example}[definition]{Example:}}%

\newenvironment{proof}{{\bf Proof:}}
{\qquad \hspace*{\fill} $\Box$}%

\newcommand{\fg}{\mathfrak{g}}%
\newcommand{\fn}{\mathfrak{n}}%
\newcommand{\fh}{\mathfrak{h}}%

\newcommand{\inner}{\operatorname{int}}%
\newcommand{\cl}{\operatorname{cl}}%
\newcommand{\fix}{\operatorname{fix}}%
\newcommand{\rme}{\mathrm{e}}%
\newcommand{\OC}{\mathcal{O}}%
\newcommand{\UC}{\mathcal{U}}%
\newcommand{\ZC}{\mathcal{Z}}%
\newcommand{\XC}{\mathcal{X}}%
\newcommand{\DC}{\mathcal{D}}%

\newcommand{\HB}{\mathbb{H}}%
\newcommand{\N}{\mathbb{N}}%
\newcommand{\R}{\mathbb{R}}%
\newcommand{\Z}{\mathbb{Z}}%

\usepackage[pagewise]{lineno}

\begin{document}

\title{Weak condition for the existence of control sets with a nonempty interior for linear control systems on nilpotent groups}
\author{Adriano Da Silva\thanks{Supported by Proyecto UTA Mayor Nº 4768-23} \\
	Departamento de Matem\'atica,\\Universidad de Tarapac\'a - Sede Iquique, Chile
 \and
 		Anderson F. P. Rojas \\
 		Instituto de Matem\'{a}tica\\
		Universidade Estadual de Campinas, Brazil}
\date{\today}
\maketitle

\begin{abstract}
	In this paper, we show that for a linear control system on a nilpotent Lie group, the Lie algebra rank condition is enough to assure the existence of a control set with a nonempty interior, as soon as the set of singularities of the drift is compact. Moreover, this control set is unique and contains the singularities of the drift in its closure. 
\end{abstract}

 {\small {\bf Keywords:} Linear control systems, control sets, nilpotent Lie groups}
	
	{\small {\bf Mathematics Subject Classification (2020): 22E25, 93C05} }%

\section{Introduction}

In the study of control systems, the control sets play an important role since they contain several dynamical properties of the systems, such as equilibrium points, recurrent points, periodic and bounded orbits, etc. In particular, exact controllability holds in the interior of such control sets.   

On the other hand, due to the work of P. Jouan \cite{JPh1}, any control-affine system with complete vector fields acting transitively on a connected manifold is equivalent to a linear control system on a Lie group or on a homogeneous space, as soon as the Lie algebra they generate is finite dimensional. Therefore, linear control systems appear as a classifying family for more general control-affine systems on connected manifolds, showing that the understanding of such systems is worth pursuing.

In this paper, we show that if the drift of a linear control system is regular (see Definition \ref{regular}), the Lie algebra rank condition is enough to assure the existence and uniqueness of a control set with a nonempty interior. This is a remarkable fact, since the Lie algebra rank condition is the weakest necessary condition to assure the existence of such subsets. 

The paper is divided as follows: In Section 2, we introduce the basic concepts of linear vector fields and linear control systems. In this section, we also study a family of diffeomorphisms induced by a linear vector field. Their understanding will be central to our main result. In Section 3, we prove our main result, namely, that a linear control system whose drift has a compact set of singularities admits a unique control set with a nonempty interior. We start by analyzing the regular case. By using the diffeomorphisms induced by the drift of the system, we first show the existence of control sets with a nonempty interior. In sequence, we show that all of the previous control sets have nonempty intersections and, hence, coincide. For the general case, we consider a fibration between nilpotent groups, whose fiber is compact and given by the set of singularities of the drift. Using general results, we are then able to show that the preimage of the control set obtained for the regular case is in fact a control set for the initial system, implying its existence and uniqueness. In Section 4, we provide two beautiful examples of non-regular systems on the Heisenberg group whose dynamics are drastically different. Both systems satisfy the Lie algebra rank condition; however, one of them is controllable and the other has no control sets with a nonempty interior but only one-point control sets.

\subsection*{Notations}

Let $N$ be a connected Lie group with Lie algebra $\fn$. By $\exp:\fn\rightarrow N$, we denote the exponential map of $N$. For any element $x\in N$, we denote by $L_x$ and $R_x$ the left and right translations of $N$ by $x$, respectively. The group of automorphisms of $N$ is denoted by $\mathrm{Aut}(N)$ and by $\mathrm{Aut}(\fn)$, the set of automorphisms of $\fn$. If $f:N\rightarrow N$ is a differentiable map, its differential at a point $x\in N$ is denoted by $(df)_x$. The set of fixed points of $f$ is the set 
$$\fix(f):=\{x\in X, f(x)=x\}.$$ 
If $X$ is a complete vector field on $N$,  with associated flow $\{\psi_S\}_{S\in\R}$, then 
$$\fix(\psi):=\bigcap_{S\in\R}\fix(\psi_S),$$
is the set of singularities of $X$.

\section{Background}

In this section, we give some background needed to understand the problem at hand. For more on the subject, the reader can consult, for instance, \cite{VATi, DS, JPh1, Jouan11}.

\subsection{Linear vector fields}

Let $N$ be a connected nilpotent Lie group, and  $\mathfrak{n}$ its associated nilpotent Lie algebra, which we identify with the set of right-invariant vector fields on $N$. A {\bf linear }vector field $\mathcal{X}$ on $N$ is a complete vector field whose associated flow $\{\varphi_S\}_{S\in\R}$ is a $1$-parameter group of automorphisms, that is, $\varphi_S\in\mathrm{Aut}(N)$ for all $S\in\R$. 
Associated with $\XC$, we have a derivation $A$ of $\fn$ defined by $A(Y):=[\XC, Y]$. The relation between $A$ and $\XC$ is given by the equations
\begin{equation}
    \label{relation}(d\varphi_{S})_{e}=\mathrm{e}^{SA}\hspace{.5cm}\mbox{ and }\hspace{.5cm} \varphi_S(\exp X)=\exp (\rme^{SA}X),
\end{equation}
where $\rme^{tA}$ is the matrix exponential of $A$.

The next result associates a diffeomorphism with any automorphism with only the origin as a fixed point. Since it is a slightly generalization of \cite[Proposition 2.7]{DSAy}, we omit its proof.

\begin{proposition}
	\label{difeo}  
 
 If $\psi\in\mathrm{Aut}(N)$ is such that $\fix((d\psi)_e)=\{0\}$, the map
	$$
	f_{\psi}: N\rightarrow N, \hspace{1cm} f_{\psi}(x):=x\psi(x^{-1}),$$
is a surjective local diffeomorphism. Moreover, $f_{\psi}$ is injective if and only if $\fix(\psi)=\{e\}$.
\end{proposition}

%\begin{proof}
%	By the very definition the map $f_{\psi}$ is differentiable. Moreover, 
%	$$
%	f_{\psi}(xy)=xy\psi((xy)^{-1})=x\left(y\psi(y^{-1})\right)
%	\psi(x^{-1})=xf_{S}(y)\psi(x^{-1}),$$
% implying that $f_{\psi}\circ L_{x}=L_{x}\circ
%	R_{\psi(x^{-1})}\circ f_{\psi}$. Since right and left-translations are diffeomorphims, we conclude that the rank of $f_{\psi}$ is constant.
% On the other hand, differentiation at the identity element of $U$ gives us that 
%	$$
%	(df_{\psi})_{e}X=\frac{d}{dt}_{|t=0}\mathrm{e}^{tX}\psi\left( \mathrm{e}%
%	^{-tX}\right) =\left( (dL_{\mathrm{e}^{tX}})_{\psi(\mathrm{e}^{-tX})}\frac
%	{d}{dt}\psi(\mathrm{e}^{-tX})+(dR_{\psi(\mathrm{e}^{-tX})})_{\mathrm{e}^{tX}%
%	}\frac{d}{dt}\mathrm{e}^{tX}\right) _{|t=0}=X-(d\psi)_eX,
%	$$
%showing that $(df_{\psi})_{e}=1-(d\psi)_e$. Therefore, 
%$$\mbox{rank}(f_{\psi})=\dim U\hspace{.5cm}\iff\hspace{.5cm}\det (df_{\psi})_{e}=\det(1-(d\psi)_e)\neq 0.$$
%Since by assumption, $\psi$ has no fixed points, we conclude that $(d\psi)_e$ also has no fixed points, which by the previous imply that $\psi$ is a local diffeomorphism. 

%On the other hand, 
%$$f_{\psi}(x)=f_{\psi}(y)\hspace{.5cm}\iff\hspace{.5cm} x\psi(x^{-1})=y\psi(y^{-1})\hspace{.5cm}\iff\hspace{.5cm} \psi(xy^{-1})=xy^{-1}\hspace{.5cm}\iff\hspace{.5cm}x=y,$$
%showing that $f_{\psi}$ is injective. The surjectivity of $f_{\psi}$ is proved by induction on the dimension of $U$. If $U$ is an abelian group, $f_{\psi}$ is an homomorphism
	
%\end{proof}

\begin{definition}
\label{regular}
    A linear vector field $\XC$ on $N$ is said to be {\bf regular} if its associated derivation $A$ satisfies $\det A\neq 0$.
\end{definition}

In what follows, we specialize on the previous result for the automorphisms $\varphi_S$ associated with a linear vector field $\XC$. 

\begin{proposition}
\label{difeolinear}
    Let $\XC$ be a linear vector field on $N$, and consider the maps $f_S:=f_{\varphi_S}$, where $\{\varphi_S\}_{S\in\R}$ is the flow of $\XC$. It holds
    \begin{enumerate}
        \item If $A$ is the derivation associated with $\XC$ then $$\fix(\varphi)=\exp\ker A.$$

        \item If $\XC$ is regular, then $N$ is simply connected;
        
        \item $f_{S_0+S_1}(x)=f_{S_0}(x)\varphi_{S_0}(f_{S_1}(x))=f_{S_1}(x)\varphi_{S_1}(f_{S_0}(x))$, for any $S_1, S_2\in\R$, $x\in N$;
        
        \item There exists a discrete subset $\ZC$ of $\R$ such that $f_S$ is a diffeomorphism for any $S\in\R\setminus\ZC$
        
    \end{enumerate}
\end{proposition}

\begin{proof} 1. The inclusion $\exp \ker A\subset \fix(\varphi)$ is straightforward and hence we will omit its proof. 

Let $x\in\fix(\varphi)$. Since, the exponential is a covering map, there exists $X\in\fn$ such that $x=\exp X$. In particular, 
$$\forall S\in\R, \hspace{.5cm}x=\exp(X)=\varphi_S(\exp X)=\exp(\rme^{SA}X)\hspace{.5cm}\implies\hspace{.5cm}\rme^{SA}X\in\exp^{-1}(\exp X)=\exp^{-1}(x).$$
However, the fact that $\exp^{-1}(x)$ is discrete for covering maps, allows us to conclude that
$$\forall S\in\R, \hspace{.5cm} \rme^{SA}X=X\hspace{.5cm}\iff\hspace{.5cm}X\in\ker A,$$
showing that $\fix(\varphi)\subset\exp \ker A$, and proving the item.

    2. Let $\widetilde{N}$ be the simply connected covering of $N$ and $\pi:\widetilde{N}\rightarrow N$ be the canonical projection between them. Consider the lifting $\widetilde{\XC}$ of $\XC$ to $\widetilde{N}$, whose associated derivation is also $A$ (see \cite[Theorem 2.2]{VATi}). In particular, the respective flows satisfy 
$$\forall S\in\R, \hspace{.5cm}\pi\circ\widetilde{\varphi}_S=\varphi_S\circ\pi,$$
implying that  $\widetilde{\varphi}_S(\ker\pi)=\ker\pi$ for all  $S\in\R$. Since $\ker\pi$ is a discrete subgroup, we conclude that 
$$\ker\pi\subset \fix(\widetilde{\varphi}).$$
On the other hand, the fact that $\widetilde{N}$ is a simply connected nilpotent Lie group implies that the exponential map is a diffeomorphism, and hence, any $x\in \widetilde{N}$ can be uniquely written as $x=\rme^X$ for some $X\in\fn$. As a consequence, we get using equation (\ref{relation}) that
$$\forall S\in\R, \hspace{.5cm}\widetilde{\varphi}_S(x)=x\hspace{.5cm}\iff\hspace{.5cm}\forall S\in\R, \hspace{.5cm}\rme^{S A}X=X, \hspace{.5cm}\iff\hspace{.5cm} X\in\ker A,$$
and since $\det A\neq 0$, we conclude that $\ker\pi=\{e\}$, showing that $\pi$ is a diffeomorphism and hence $N$ is simply connected.

3. It follows directly from the definition of $f_S$. In fact, 
$$
f_{S_0+S_1}(x)=x\varphi_{S_0+S_1}(x^{-1})=x\varphi_{S_0}(x^{-1})\varphi_{S_0}(x)\varphi_{S_0}(\varphi_{S_1}(x^{-1}))=f_{S_0}(x)\varphi_{S_0}(f_{S_1}(x)).  $$

4. Let us consider $\pm i\mu_j, j=1, \ldots, m$ be the pure imaginary eigenvalues of $A$ and define the discrete subset of $\R$ given by
$$\ZC:=\left\{\frac{2k\pi}{\mu_j}, j=1, \ldots, m, k\in\Z\right\},$$
where $\ZC:=\{0\}$ if $A$ has no pure imaginary eigenvalue.  By Proposition \ref{difeo}, $f_S$ is a diffeomorphism as soon as $\fix(\rme^{S A})=\{0\}$ and $\fix(\varphi_S)=\{e\}$. Moreover, since $N$ is simply connected, relation (\ref{relation})
implies that 
$$\varphi_S(x)=x\hspace{.5cm}\iff \hspace{.5cm}\rme^{SA}X=X,\hspace{.5cm}\mbox{ with }\hspace{.5cm} x=\rme^X,$$
and hence, $f_S$ is a diffeomorphism as soon as $\det(I_{\fn}-\rme^{SA})\neq 0$ or equivalently, if $ S\in\R\setminus\ZC,$
concluding the proof.
\end{proof}

We finish this section with a lemma which assures that the set of the singularities of a linear vector field on a nilpotent Lie group $N$ is connected.

\subsection{Linear control systems}

Let $N$ be a connected Lie group with Lie algebra $\fn$ identified with the set of right-invariant vector fields on $N$.

A {\bf linear control system (abrev. LCS)} on $N$ is defined by the family of ODEs, 
\begin{flalign*}
&&\dot{x}(t)=\XC(x(t))+\sum_{j=1}^mu_j(t)Y^j(x(t)),  &&\hspace{-1cm}\left(\Sigma_N\right)
\end{flalign*}
where the {\bf drift} $\XC$ is a linear vector field, $Y^1, \ldots, Y^m\in\fn$ and ${\bf u}=(u_1, \ldots, u_m)\in\UC$. Here, $\UC$ is the set of admissible {\bf control functions} given by
$$\UC:=\{{\bf u}:\R\rightarrow\R^m; \;{\bf u}\;\mbox{ is a piecewise constant function with }  u(\R)\subset\Omega\},$$
where $\Omega\subset\R^m$ is a compact and convex subset containing the origin in its interior. We say that the LCS is {\bf regular} if its drift is regular.

For any $x\in N$ and ${\bf u}\in\UC$, the solution $S\mapsto\phi(S, x, {\bf u})$ of $\Sigma_N$ is complete and satisfies  
\begin{equation}
\label{prop}
\phi_{S, {\bf u}}\circ R_x=R_{\varphi_{S}(x)}\circ\phi_{S, {\bf u}}, \hspace{.5cm}\mbox{ for any } \hspace{.5cm}S\in\R, x\in N,
\end{equation}
where $\phi_{S, {\bf u}}(x):=\phi(S, x,  {\bf u})$, for any $x\in N$, and $\{\varphi_S\}_{S\in\R}$ is the flow of $\XC$.

For any $x\in N$, we define respectively
$$\mathcal{O}^{+}_{S}(x):=\{\phi(S,x, {\bf u}), {\bf u}\in \mathcal{U}\}\hspace{.5cm}\mbox{ and }\hspace{.5cm}\mathcal{O}^{+}(x):=\bigcup_{S>0}\mathcal{O}^{+}_{S}(x),$$
as the {\bf reachable from $x$} in time $S>0$ and the {\bf reachable set} of $x$. We say that $\Sigma_{N}$ satisfy the {\bf Lie algebra rank condition (abrev. LARC)} if $\fg$ is the smallest $\DC$-invariant Lie subalgebra containing $Y^1, \ldots, Y^m$. In particular, if the LARC is satisfied, the sets 
         $$\mathcal{O}^{+}_{\leq S}(x):=\bigcup_{0<s\leq S}\mathcal{O}^{+}_{s}(x), \hspace{.5cm}S>0,$$ have a nonempty interior. Moreover, $\inner\mathcal{O}^{+}_{\leq S}(x)$ (resp. $\inner\mathcal{O}^{+}(x)$) is dense in $\mathcal{O}^{+}_{\leq S}(x)$ (resp. $\mathcal{O}^{+}(x)$) for all $S>0$, $x\in N$.

The next proposition, \cite[Proposition 2]{Jouan11}, states the main properties of the reachable sets of a LCS.  

\begin{proposition}
    \label{propertiesLCSs}
    It holds:
    \begin{enumerate}
        \item $\mathcal{O}^{+}_{\leq S}(e)=\mathcal{O}^{+}_{S}(e)$ for all $S>0$;

        \item  $\mathcal{O}^{+}_{S_1}(e)\subset \mathcal{O}^{+}_{S_2}(e)$ for all $0<S_1\leq S_2$;

        \item $\mathcal{O}^{+}_{S}(x)=\mathcal{O}^{+}_{S}(e)\varphi_S(x)$ for all $S>0$ and $x\in N$;

        \item $\mathcal{O}^{+}_{S_1+S_2}(e)=\mathcal{O}^{+}_{S_1}(e)\varphi_{S_1}\left(\mathcal{O}^{+}_{S_2}(e)\right)=\mathcal{O}^{+}_{S_2}(e)\varphi_{S_2}\left(\mathcal{O}^{+}_{S_1}(e)\right)$, for all $S_1, S_2>0$.
    \end{enumerate}
\end{proposition}

\begin{remark}
\label{density}
    The previous proposition implies, in particular, that under the LARC, it holds
    $$\cl\left(\inner\mathcal{O}^{+}_{S}(x)\right)=\cl\left(\mathcal{O}^{+}_{S}(x)\right), \hspace{.5cm}\forall S>0, x\in N.$$
    
\end{remark}

Next, we define the concept of control sets.

\begin{definition}
\label{controlset}

A subset $D$ of $M$ is called a control set of the system $\Sigma_M$ if it satisfies the following properties:

\begin{itemize}
    \item [(i)] {\it (Weak invariance)} For every $x\in M$, there exists a ${\bf u}\in\mathcal{U}$ such that $\phi(\mathbb{R}^+,x,{\bf u})\subset D$;
    \item [(ii)] {\it (Approximate controllability)}  $ D\subset\mathrm{cl}\left(\mathcal{O}^{+}(x)\right)$ for every $x\in D$;
    \item[(iii)] {\it (Maximality)} $D$ is maximal with respect to properties (i) and (ii).
\end{itemize}
\end{definition}

In particular, when the whole state space $M$ is a control set, $\Sigma_M$ is said to be {\it controllable}.

\begin{remark}

It is worth mentioning that several results for control sets of LCSs were obtained in \cite{DSAyGZ} under the (much stronger) condition of local controllability around the identity element of the group. Such a condition, although sufficient, is not necessary for the existence of a control set with a nonempty interior. In fact, examples of LCSs with the identity element on the boundary of a control set with a nonempty interior can be found in \cite{DS1}.
\end{remark}

Let us finish this section with some comments on conjugations of linear systems.

Suppose $\widehat{N}$ is another nilpotent Lie group, and 
\begin{flalign*}
	  &&\dot{z}(s) = \hat{\XC}(z(s)) + \sum_{j=1}^{m}u_j(s)\hat{Y}^j(z(s)),  \hspace{.5cm} {\bf u} = (u_1,\ldots, u_m)\in \mathcal{U}. &&\hspace{-1cm}\left(\Sigma_{\widehat{N}}\right)
	  \end{flalign*}
a LCS on $\widehat{N}$. If $\psi:N\rightarrow \widehat{N}$ is a surjective homomorphism, we say that $\Sigma_N$ and $\Sigma_{\widehat{N}}$ are $\psi$-conjugated if their respective vector fields are $\psi$-conjugated, that is, 

$$\psi_*\circ\XC=\hat{\XC}\circ\psi\hspace{.5cm}\mbox{ and }\hspace{.5cm}\;\;\;\psi_*\circ Y^j=\hat{Y}^j\circ\psi, \hspace{.5cm} j=1,\ldots m.$$

If such $\psi$ exists, we say that $\Sigma_N$ and $\Sigma_{\widehat{N}}$ are $\psi$-conjugated. If $\psi$ is an automorphism, $\Sigma_N$ and $\Sigma_{\widehat{N}}$ are said to be equivalent. The next result, whose prove we omit, relates control sets of conjugated systems. 

\begin{proposition}
\label{conjugation}
    Let $\Sigma_N$ and $\Sigma_{\widehat{N}}$ be $\psi$-conjugated LCSs. It holds:
\begin{enumerate}
    \item[1.] If $D$ is a control set of $\Sigma_N$, there exists a control set $\widehat{D}$ of $\Sigma_{\widehat{N}}$ such that $\psi(D)\subset \widehat{D}$;
    \item[2.] If for some $y_0\in \inner \widehat{D}$ it holds that $\psi^{-1}(y_0)\subset\inner D$, then $D=\psi^{-1}(\widehat{D})$.
\end{enumerate}
    
\end{proposition}

\section{Control sets of LCSs}

In this section, we show that the LARC is enough to assure the existence of control sets with nonempty interiors for LCSs whose associated drift has compact set of singularities. Moreover, we show that such a control set is unique and contains the identity element in its boundary.

We start with a technical lemma which will help us ahead.

\begin{lemma}
\label{manycontrols}
    Let $\gamma:[0, 1]\rightarrow N$ be a continuous path and assume that for any $t\in [0, 1]$, there exists a control set $D_{\gamma(t)}$ of $\Sigma_N$ such that 
    $\gamma(t)\in \inner D_{\gamma(t)}$. Then $D_{\gamma(t)}=D_{\gamma(0)}$ for all $t\in [0, 1]$.
\end{lemma}

\begin{proof} For any $t\in [0, 1]$, there exists, by continuity, $\varepsilon>0$ such that $\gamma\left([t, t+\varepsilon)\right)\subset \inner D_{\gamma(t)}$. Consequently, if 
$$T_0:=\max\{T\in [0, 1]; \;D_{\gamma(t)}=D_{\gamma(0)}, \forall t\in [0, T]\},$$
then necessarily $T_0=1$, which shows the result.
\end{proof}

\subsection{The regular case}

In this section we prove the existence of a unique control set with nonempty interior for regular systems.

\begin{theorem}
\label{main1}
    A regular LCS on a nilpotent connected Lie group $N$ satisfying the LARC admits a unique control set with a nonempty interior.
\end{theorem}

\begin{proof} Let $\Sigma_N$ be a regular LCS on $N$ with drift $\XC$. Denote as previously by $\{\varphi_S\}_{S\in \R}$ the flow of $\XC$. By Proposition \ref{difeo}, there exists an open and dense set $I:=\R\setminus \ZC$  such that
$$\forall S\in I, \hspace{.5cm }x\in N\mapsto f_S(x)=x\varphi_S(x^{-1})\hspace{.5cm}\mbox{ is a diffeomorphism.}$$
The rest of the proof is divided into the next four steps.

{\bf Step 1:} For any $S\in I$ and $x\in f_S^{-1}\left(\inner\OC_S^+(e)\right)$ there exists a control set $D_{S, x}$ satisfying $x\in\inner D_{S, x}$. 

By the LARC, it holds that 
$$\forall S>0, \hspace{1cm} \inner\OC^+_{S}(e)\neq\emptyset.$$
Therefore, for all $S\in I$ and $y\in \inner\OC^+_{S}(e)$,  there exists $x\in N$ such that $f_{S}(x)=y$. Writing $y=\phi(S, e, {\bf u})$ we get  
$$x\varphi_{S}(x^{-1})=f_{S}(x)=y=\phi(S, e, {\bf u})\hspace{.5cm}\iff\hspace{.5cm}x=\phi(S, e, {\bf u})\varphi_{S}(x)\in \inner\OC^+_{S}(e)\varphi_{S}(x)=\inner\OC^+_{S}(x),$$
showing, in particular, that $\OC^+(x)$ is open. By general results (see for instance \cite[Chapter 3]{FCWK}), there exist a control set $D_{S, x}$ satisfying $x\in\inner D_{S, x}$, showing step 1.

{\bf Step 2:} For any $S\in I$, there exists a control set $D_S$ satisfying 
$$f_S^{-1}\left(\inner\OC_S^+(e)\right)\subset \inner D_S.$$

By the previous step, it is enough to show that the control sets $D_{S, x}$ are independent of $x\in f_S^{-1}\left(\inner\OC_S^+(e)\right)$. 

However, the fact that $\inner\OC_S^+(e)$ is connected (\cite[Theorem 3]{Jurd}) and that, for any $S\in I$, the map $f_S$ is a diffeomorphism, imply that the set  
        $f_S^{-1}\left(\inner\OC_S^+(e)\right)$ is open and connected, and hence path connected. Consequently, for any $x, y\in f_S^{-1}\left(\inner\OC_S^+(e)\right)$, there exists a continuous path $\gamma:[0, 1]\rightarrow N$ such that $\gamma(0)=x$, $\gamma(1)=y$ and $\gamma(t)\in f_{S}^{-1}\left(\inner\OC_S^+(e)\right)$ for any $t\in [0, 1]$. By Step 1, $\gamma(t)\in \inner D_{S, \gamma(t)}$ which by Lemma \ref{manycontrols} implies necessarily  $D_{S, x}=D_{S, y}$, and hence 
        $$D_S=D_{S, x}, \hspace{.5cm}\forall  x\in f_S^{-1}\left(\inner\OC_S^+(e)\right),$$
        showing the step 2.

{\bf Step 3:} There exists a control set $D$ such that 
$$\forall S\in I, \hspace{1cm}f_S^{-1}\left(\inner\OC_S^+(e)\right)\subset \inner D.$$

By the previous step, we only have to show that $D_{S_0}=D_{S_1}$ for any $S_0, S_1\in I$. Assume first that $[S_0, S_1]\subset I$ and let $x\in \inner\OC_{S_0}^+(e)$. Define the curve
$$\gamma:[0, 1]\rightarrow N, \hspace{1cm} \gamma(t):=f^{-1}_{tS_1+(1-t)S_0}(x).$$
Since for all $t\in (0, 1)$ we have that $tS_1+(1-t)S_0\in I$, we get that
$$\gamma(t)=f_{tS_1+(1-t)S_0}^{-1}(x)\in f_{tS_1+(1-t)S_0}^{-1}\left(\inner\OC_{S_0}^+(e)\right)\subset f_{tS_1+(1-t)S_0}^{-1}\left(\inner\OC_{tS_1+(1-t)S_0}^+(e)\right)\subset \inner D_{tS_1+(1-t)S_0},$$
where for the first inclusion we used Proposition \ref{propertiesLCSs}. By Lemma \ref{manycontrols}, we conclude that $D_{S_0}=D_{S_1}$, showing that control sets $D_S$ depend only on the connected components of $I$.

Let $(0, \varepsilon)$ be a connected component of $I$, and denote by $D$ the control set satisfying
$$\forall S\in (0, \varepsilon), \hspace{1cm} f_S^{-1}\left(\inner\OC_S^+(e)\right)\subset D.$$
Take an arbitrary $S\in I$ and choose $n\in \N$ large enough, such that $S<n\varepsilon$, and consider  
$$x\in f_{\frac{S}{n}}^{-1}\left(\inner\OC^+_{\frac{S}{n}}(e)\right)\subset \inner D.$$
Using Proposition \ref{difeolinear}, we get 
$$f_{2\frac{S}{n}}(x)=f_{\frac{S}{n}}(x)\varphi_{\frac{S}{n}}\left(f_{\frac{S}{n}}(x)\right)\in \inner\OC^+_{\frac{S}{n}}(e)\varphi_{\frac{S}{n}}\left(\inner\OC^+_{\frac{S}{n}}(e)\right)=\inner\OC^+_{2\frac{S}{n}}(e),$$
$$f_{3\frac{S}{n}}(x)=f_{\frac{S}{n}}(x)\varphi_{\frac{S}{n}}\left(f_{2\frac{S}{n}}(x)\right)\in \inner\OC^+_{\frac{S}{n}}(e)\varphi_{\frac{S}{n}}\left(\inner\OC^+_{2\frac{S}{n}}(e)\right)=\inner\OC^+_{3\frac{S}{n}}(e),$$
and inductively,
$$f_S(x)=f_{n\frac{S}{n}}(x)=f_{\frac{S}{n}}(x)\varphi_{\frac{S}{n}}\left(f_{(n-1)\frac{S}{n}}(x)\right)\in \inner\OC^+_{\frac{S}{n}}(e)\varphi_{\frac{S}{n}}\left(\inner\OC^+_{(n-1)\frac{S}{n}}(e)\right)=\inner\OC^+_{n\frac{S}{n}}(e)=\inner\OC^+_S(e).$$
Consequently, 
$$ x\in f_{S}^{-1}\left(\inner\OC^+_S(e)\right)\subset \inner D_S,$$
which implies $D_{S}=D$ for any $S\in I$, showing step 3. 

{\bf Step 4:} $D$ is the only control set of $\Sigma_N$.

By the previous items, we only have to show that any arbitrary control set $D'$ with a nonempty interior satisfies
$$\inner D'\cap f_{S}^{-1}\left(\inner\OC^+_S(e)\right)\neq \emptyset
, \hspace{.5cm}\mbox{ for some }\hspace{.5cm}S\in I.$$
  
   Let us consider $x\in \inner D'$ and $S_0>0$. Since $\inner \OC^+_{S_0}(x)$ is dense in $\OC^+_{S_0}(x)$ (see Remark \ref{density}), we get
    $$\inner \OC^+_{S_0}(x)\cap\inner D'\neq\emptyset.$$
    For $y$ in this intersection, there exists, by exact controllability, $S_1>0$ and ${\bf u}\in \UC$ such that 
    $\phi(S_1, y, {\bf u})=x$ and hence
    $$x=\phi(S_1, y, {\bf u})\in \phi_{S_1, {\bf u}}(\inner \OC^+_{S_0}(x))\subset \inner \OC^+_{S_0+S_1}(x)=\inner \OC^+_{S_0+S_1}(e)\varphi_{S_0+S_1}(x),$$
    implying that
    $$x\varphi_{S_0+S_1}(x^{-1})\in \inner \OC^+_{S_0+S_1}(e).$$
    On the other hand, the fact that 
$I$ is dense in $(0, +\infty)$ assures, by continuity, the existence of $S\in I$, with $S>S_0+S_1$, such that 
    $$x\varphi_S(x^{-1})\in\inner \OC^+_{S_0+S_1}(e)\subset \inner \OC^+_S(e),$$
    and so, $$f_S(x)=x\varphi_S(x^{-1})\in \inner\OC^+_S(e)\hspace{.5cm}\iff\hspace{.5cm} x\in f_S^{-1}(\inner\OC^+_S(e)),$$
    which implies step 4. and concludes the proof of the theorem.
\end{proof}

As a direct corollary, we get the following.

\begin{corollary}
\label{coro}
    Under the assumptions of the previous theorem, the only control set $D$ of a LCS $\Sigma_N$ on $N$ satisfies $e\in \mathrm{cl}(D)$.
\end{corollary}

\begin{proof}
    In fact, since $f_S(e)=e$, for any $S\in I$, we get that
    $$e=f_S^{-1}(e)\in f_S^{-1}\left(\mathrm{cl}\left(\OC_S^+(e)\right)\right)=f_S^{-1}\left(\mathrm{cl}\left(\inner\OC_S^+(e)\right)\right)=\mathrm{cl}\left(f_S^{-1}\left(\inner\OC_S^+(e)\right)\right)\subset \mathrm{cl}(D).$$
\end{proof}

\subsection{The general case}

In this section we conclude the general case where the set of singularities of the drift is a compact subset.

\begin{theorem}
    A LCS on a nilpotent Lie group $N$ satisfying the LARC and whose associated drift has compact set of singularities admits a unique control set with nonempty interior. Moreover, its closure contains the singularities of the drift.
\end{theorem}

\begin{proof}
Let $\Sigma_N$ be a LCS whose drift $\XC$ has compact set of singularities, which we denote here by $H$. Since compact subgroups of nilpotent Lie groups are central (see \cite[Theorem 1.6]{ALEB}) it holds that $H\subset Z(N)$. Consequently, $\widehat{N}:=N / H$ is a nilpotent Lie group. Moroever, the fact that $H$ is invariant by the flow of $\XC$ assures the existence of a LCS $\Sigma_{\widehat{N}}$ on $\widehat{N}$ that is $\pi$-conjugated to $\Sigma_N$, where $\pi:N\rightarrow \widehat{N}$ is the canonical projection (see \cite[Proposition 4]{JPh1}). In particular, the derivation $A'$ associated with the drift $\widehat{\XC}$ of $\Sigma_{\widehat{N}}$ satisfies
    $$(d\pi)_{e}\circ A=A'\circ (d\pi)_{e}\hspace{.5cm}\implies\hspace{.5cm}\det A'\neq 0,$$
    where $A$ is the derivation associated with $\XC$.

    By Theorem \ref{main1}, the LCS $\Sigma_{\widehat{N}}$ admits a unique control set $\widehat{D}$ with nonempty interior. 

    Let $x\in\pi^{-1}(\inner \widehat{D})$ and consider, as in Step 4 of the previous result, $S>0$ such that $\pi(x)\in\inner\widehat{\OC}_S^+(\pi(x))$. By the conjugation property,
    $$\pi^{-1}\left(\inner\widehat{\OC}_S^+(\pi(x))\right)=\bigcup_{g\in H}\inner\OC_S^+(xg)=\bigcup_{g\in H}\inner\OC_S^+(x)g,
    $$
    and hence, there exists $g\in H$ satisfying $xg\in \inner\OC_S^+(x)$. Moreover, the fact that the elements of finite order are dense in a compact group, together with the openness of $\inner\OC_S^+(x)$, allows us to assume w.l.o.g. that $$xg\in \inner\OC_S^+(x), \hspace{.5cm}\mbox{ with }\hspace{.5cm}g^n=e, \hspace{.5cm}\mbox{ for some }\hspace{.5cm}n\in \N.$$
    Let then ${\bf u}\in\UC$ such that $\phi(S, x, {\bf u})=xg$ and extend it $S$-periodically. By concatenation, one gets that
    $$\forall k\in\N, \hspace{.5cm}\phi(kS, x, {\bf u})=xg^k\hspace{.5cm}\implies\hspace{.5cm}x=\phi(nS, x, {\bf u})\in\inner\OC_{nS}^+(x),$$
which assures the existence of a control set $D_x$ with $x\in\inner D_x$. 

Since the previous holds for any $x\in\pi^{-1}(\inner\widehat{D})$, for any $h\in H$ there exists a control set $D_{xh}$ with $xh\in\inner D_{xh}$. However, the fact that $H$ is a connected Lie group (see Proposition \ref{difeolinear}) implies that, for any $h_1, h_2\in H$, there exists a continuous curve $\gamma:[0, 1]\rightarrow H$ such that $\gamma(0)=h_1$, $\gamma(1)=h_2$ and $$x\gamma(t)\in xH=\pi^{-1}(\pi(x))\subset\pi^{-1}(\inner\widehat{D}),\hspace{.5cm}\forall t\in [0, 1].$$

Therefore, for any $t\in[0, 1]$ there exists a control set $D_{x\gamma(t)}$ with $x\gamma(t)\in\inner D_{x\gamma(t)}$, which by Lemma \ref{manycontrols} implies that $D=D_{x\gamma(t)}$ for all $t\in [0, 1]$. Therefore, $D$ is a control set of $\Sigma_N$ with nonempty interior, such that 
$$xH=\pi^{-1}(\pi(x))\subset\inner D\hspace{.5cm}\mbox{ and }\hspace{.5cm} \pi(x)\in \inner \widehat{D},$$
which by Proposition \ref{conjugation} implies that $D=\pi^{-1}(\widehat{D})$. Uniqueness follows direct from the previous equality. Also, by Corollary \ref{coro} and the fact that $\pi$ is a closed map, we get that
$$\pi(H)\subset \mathrm{cl}\left(\widehat{D}\right)\hspace{.5cm}\implies\hspace{.5cm} H\subset\pi^{-1}\left(\mathrm{cl}\left(\widehat{D}\right)\right)=\mathrm{cl}\left(\pi^{-1}(\widehat{D})\right)=\mathrm{cl}\left(D\right),$$
concluding the proof.
\end{proof}

\section{LCSs on the Heisenberg group}

In this section we give some examples of non-regular LCSs on the Heisenberg group.

The {\bf Heisenberg group} $\HB$ is by definition $\HB:=(\R^3, *)$ where the product is given by 
$$(x_1, y_1, z_1)*(x_2, y_2, z_2):=\left(x_1+x_2, y_1+y_2, z_1+z_2+\frac{1}{2}(x_2y_1-x_1y_2)\right).$$
Its Lie algebra $\fh$ is 
$\fh:=(\R^3, [\cdot, \cdot])$ where
$$[(\alpha_1, \beta_1, \gamma_1), (\alpha_2, \beta_2, \gamma_2)]:=\left(0, 0, \alpha_2\beta_1-\alpha_1\beta_2\right).$$

\begin{example} {\it Non-regular LCS on $\HB$ with an infinite number of one-point control sets}

The control-affine system 
\begin{flalign*}
&&\left \{\begin{array}
[c]{l}%
\dot{x}(S)=y(S)\\
\dot{y}(S)={\bf u}(S)\\
\dot{z}(S)=\frac{1}{2}x(S){\bf u}(S)
\end{array}\right., \hspace{.5cm}\mbox{ where }\Omega:=[\rho, \nu], \hspace{.5cm}\rho<0<\nu, &&\hspace{-1cm}\left(\Sigma_{\HB}\right)
\end{flalign*}
is a LCS on $\HB$ whose associated derivation and right-invariant vector field are given, respectively, as 
$$A=\left(\begin{array}{ccc}
    0 & 1 & 0 \\
    0 & 0 & 0 \\
    0 & 0 & 0
\end{array}\right)\hspace{.5cm}\mbox{ and }\hspace{.5cm}Y(x, y, z)=\left(0, 1, \frac{1}{2}x\right).$$
In particular, if $Y=Y(0, 0, 0)$ we have that
$$AY=(1, 0, 0)\hspace{.5cm}\mbox{ and }\hspace{.5cm} [AY, Y]=(0, 0, 1)\hspace{.5cm}\implies\hspace{.5cm}\fh=\mathrm{span}\{Y, AY, [AY, Y]\},$$
and $\Sigma_{\HB}$ satisfies the LARC. Moreover, $\det A=0$ and the system is non-regular. 

The plane $S=\{(x, 0, z), x, z\in\R\}$ is the set of singularities of the drift. In particular, any point in $S$ has to be contained in a control set of $\Sigma_{\HB}$. Our aim here is to show that any point in $S$ is in fact a control set of $\Sigma_{\HB}$. 

In order to show the previous, let us consider the diffeomorphism, 
$$f:\HB\rightarrow\R^3, \hspace{1cm}f(x, y, z)=\left(x, y, z-\frac{1}{6}xy\right),$$
we get that $\Sigma_{\HB}$ is equivalent to the system \footnote{The change from $\HB$ to $\R^3$ in the subscripts is to emphasize that $\Sigma_{\R^3}$ is not a linear control system.}
\begin{flalign*}
&&\left \{\begin{array}
[c]{l}%
\dot{x}(S)=y(S)\\
\dot{y}(S)={\bf u}(S)\\
\dot{z}(S)=\frac{1}{6}(2x(S){\bf u}(S)-y(S)^2)
\end{array}\right., &&\hspace{-1cm}\left(\Sigma_{\R^3}\right)
\end{flalign*}
whose solution for a constant control function ${\bf u}\equiv u$ is
$$\phi(S, P_0, {\bf u})=\left(x_0+y_0S+\frac{uS^2}{2}, y_0+Su, z_0+\frac{1}{6}(2x_0u-y_0^2)S\right).$$

Let us consider the function 
$$F:\R^3\rightarrow\R, \hspace{1cm}F(x, y, z):=6z\sigma+y(y^2-2x\sigma),$$
where $\sigma<\rho$ is a fixed real number. For any $u\in \Omega$ and $P_0=(x_0, y_0, z_0)$, it holds that:

If $u=0$, then 
$$F(\phi(S, P_0, 0))=F(x_0+Sy_0, y_0, z_0)=6z_0\sigma+y_0(y_0^2-2(x_0+ty_0)\sigma)$$
$$=F(P_0)-2Sy_0^2\sigma\geq F(P_0)
,\hspace{.5cm}\mbox{ since }\hspace{.5cm} \sigma<\rho<0,$$
with equality iff $y_0=0$.

On the other hand, if $u\neq 0$, then  $$g_u(S):=F(\phi(S, P_0, u))=6\left(z_0+\frac{1}{6}(2x_0u-y_0^2)\right)\sigma+(y_0+uS)\left((y_0+uS)^2-2\left(x_0+y_0t+\frac{uS^2}{2}\right)\sigma\right)$$
$$=(y_0+uS)^3+\sigma\left(6z_0-2x_0y_0-3y_0^2S-3y_0uS^2-u^2S^3\right)=\frac{u-\sigma}{u}(y_0+uS)^3+\sigma\left(6z_0-2x_0y_0+\frac{1}{u}y_0^3\right).$$
   Derivation in $S$, gives us that 
   $$g_u'(S)=3(u-\sigma)(y_0+uS)^2\geq 0, $$
 showing that, for all $S>0$,
 $$F(\phi(S, P_0, u))\geq F(P_0),$$
 with equality iff $y_0=0$ and $u=0$. By concatenation, we conclude that 
 $$F(\phi(S, P_0, {\bf u}))> F(P_0), \hspace{.5cm} \mbox{ if }\hspace{.5cm}y_0\neq 0\hspace{.5cm} \mbox{ or }\hspace{.5cm}{\bf u}\not\equiv 0,$$
 implying that the control sets of $\Sigma_{\R^3}$ are 
the singletons $\{(x_0, 0, z_0)\}$.
\end{example}

\begin{example} {\it Controllable non-regular LCS on $\HB$}

Let us consider the linear control system 
\begin{flalign*}
&&\left \{\begin{array}
[c]{l}%
\dot{x}(S)=-y(S)\\
\dot{y}(S)=x(S)+{\bf u}(S)\\
\dot{z}(S)=\frac{1}{2}x(S){\bf u}(S)
\end{array}\right., \hspace{.5cm}\mbox{ where }\Omega:=[\rho, \nu], \hspace{.5cm}\rho<0<\nu, &&\hspace{-1cm}\left(\Sigma_{\HB}\right)
\end{flalign*}
is a LCS on $\HB$ whose associated derivation and right-invariant vector field are given, respectively, as 
$$A=\left(\begin{array}{ccc}
    0 & -1 & 0 \\
    1 & 0 & 0 \\
    0 & 0 & 0
\end{array}\right)\hspace{.5cm}\mbox{ and }\hspace{.5cm}Y(x, y, z)=\left(0, 1, \frac{1}{2}x\right).$$
In particular, for $Y=Y(0, 0, 0)$ it holds that
$$AY=(-1, 0, 0)\hspace{.5cm}\mbox{ and }\hspace{.5cm} [AY, Y]=\left(-1, 1, -\frac{1}{2}\right)\hspace{.5cm}\implies\hspace{.5cm}\fh=\mathrm{span}\{Y, AY, [AY, Y]\},$$
showing that $\Sigma_{\HB}$ satisfies the LARC. In particular, $\det A=0$ and the system is non-regular. 

The solutions of $\Sigma_{\HB}$  for a constant control function ${\bf u}\equiv u$ is 
$$\phi(S, P_0, {\bf u})=\left((x_0+u)\cos S-y_0\sin S-u, (x_0+u)\sin S+y_0\cos S, z_0-\frac{u}{2}\left((x_0+u)\sin S+y_0(\cos S-1)\right)+\frac{u^2}{2}S\right).$$

Let us note that the Since the induced system on $\R^2$
\begin{flalign*}
&&\left \{\begin{array}
[c]{l}%
\dot{x}(S)=-y(S)\\
\dot{y}(S)=x(S)+{\bf u}(S)
\end{array}\right., &&\hspace{-1cm}\left(\Sigma_{\R^2}\right)
\end{flalign*}
is a LCS whose associated derivation has a pair of pure imaginary eigenvalues. Since $\Sigma_{\R^2}$ also satisfies the LARC, it is controllable (see for instance \cite[Chapter 3]{FCWK}). Therefore, $\Sigma_{\HB}$ is controllable, as soon as $\{0\}\times\R$ is controllable. In order to show the previous, we construct explicitly a trajectory connecting two arbitrary points $(0, z_1), (0, z_2)\in\{0\}\times\R$ with $z_1<z_2$ as follows:

\bigskip

$\bullet $ {\bf From $(0, z_1)$ to $(0, z_2)$:}

 Let $\delta>0$ such that $(-\delta, \delta)\subset(\rho, \nu)$. Take $\epsilon>0$ satisfying $3\epsilon=z_2-z_1$ and consider $S_0>0$ such that 
    $$\overline{\OC_{S_0}^+(0, z_1)}\subset B_{\epsilon}(0, z_1)\hspace{.5cm}\mbox{ and }\hspace{.5cm}\overline{\OC_{S_0}^-(0, z_2)}\subset B_{\epsilon}(0, z_2),$$
    which exists by the continuity of the solutions and the compactness of $[\rho, \nu]$. Moreover, the LARC implies that $\pi(\OC_{S_0}^+(0, z_1))$ and $\pi(\OC_{S_0}^-(0, z_2))$ are open neighborhoods of the origin in $\R^2$, where $\pi: (v, z)\in\HB\rightarrow v\in \R^2$ is the projection onto the two first components. Since $\Sigma_{\R^2}$ is controllable,  there exists $u\in(-\delta, 0)$ and ${\bf u}_1, {\bf u}_2\in\UC$ such that 
    $$\phi(S_0, (0, z_1), {\bf u}_1)=(v(u), \bar{z}_1) \hspace{.5cm}\mbox{ and }\hspace{.5cm} \phi(S_0, (v(u), \bar{z}_2), {\bf u}_1)= (0, z_2),$$
and by our choices $\bar{z}_1<\bar{z}_2$, where $v(u):=(-u, 0)$ is an equilibrium point of $\Sigma_{\R^2}$. On the other hand, 
$$S_1=\frac{\bar{z}_2-\bar{z}_1}{p(u)}>0\hspace{.5cm}\implies\hspace{.5cm}\phi(S_1, (v(u), \bar{z}_1), u)=(v(u), \bar{z}_1+S_1p(u))=(v(u), \bar{z}_2),$$
   and so, by concatenation, we get 
   $$\phi(S_0, \phi(S_1, \phi(S_0, (0, z_1), {\bf u}_1), u), {\bf u}_2)=(0, z_2).$$

\bigskip

$\bullet $ {\bf From $(0, z_2)$ to $(0, z_1)$:}

\bigskip
    Let us fix $\alpha\in(0, \nu)$ and define $v^*:=-\alpha\left(1, \pi\right)$. By the controllability of the induced system $\Sigma_{\R^2}$, there exist ${\bf u_1}, {\bf u_2}, {\bf u}^*\in\UC$, $S_1, S_2, S^*>0$ and $\bar{z}_1, \bar{z}_2, z^*\in\R$ satisfying
    $$\phi(S_2, (0, z_2), {\bf u}_2)=(v^*, \bar{z}_2), \hspace{.5cm}\phi(S_1, (v(\alpha), \bar{z}_1), {\bf u}_1)=(0, z_1)\hspace{.5cm} \mbox{ and }\hspace{.5cm}\phi(S^*, (v^*, 0), {\bf u}^*)=(v(\alpha), z^*),$$
    where as previously $v(\alpha)=(-\alpha, 0)$. A trajectory connecting $(0, z_2)$ to $(0, z_1)$ is construct by concatenation as follows: 

    \begin{itemize}
        \item[(i)] With control ${\bf u}_2$ and time $S_2$ connect $(0, z_2)$ with $(v^*, \bar{z}_2)$; 
    
        \item[(ii)] Using the constant control ${\bf u}\equiv\alpha$ and $\tau_0=\pi$ we have that 
        $$\phi\left(\pi, (v^*, \bar{z}_2), \alpha\right)=\left(-\alpha\left(1, -\pi\right), \bar{z}_2-\alpha^2\frac{\pi}{2}\right);$$

        \item[(iii)] Since  $v^*$ and $-\alpha(1, -\pi)$ have the same norm,  there exists $\tau_1>0$ such that $R_{\tau_1}\left(-\alpha(1, -\pi)\right)=v^*$, where $R_{\tau_1}$ is the counter-clockwise rotation of $\tau_1$-degrees. Hence, 
        $$\phi\left(\tau_1, \phi\left(\tau_0, (v^*, \bar{z}_2), \alpha\right), 0\right)=\phi\left(\tau_1, \left(-\alpha(1, -\pi), \bar{z}_2-\alpha^2\frac{\pi}{2}\right), 0\right)$$
        $$=\left(R_{\tau_1}\left(-\alpha(1, -\pi)\right), \bar{z}_2-\alpha^2\frac{\pi}{2}\right)=\left(v^*, \bar{z}_2-\alpha^2\frac{\pi}{2}\right),$$
        which gives us a periodic trajectory for the induced system on the plane (see Figure \ref{planar})

        \item[(iv)] Repeat items (ii) and (iii) $n_0$-times to obtain a trajectory connecting $(v^*, \bar{z}_2)$ to the point $(v^*, \hat{z}_2)$, where 
        
        $$\hat{z}_2:= \bar{z}_2 -n_0\alpha^2\frac{\pi}{2}\hspace{.5cm}\mbox{ satisfies }\hspace{.5cm}\hat{z}_2\leq \bar{z}_1-z^*;$$

        \item[(v)] Now, with control ${\bf u}^*$ and time $S^*$ we have that 
        $$\phi(S^*, (v^*, \hat{z}_2), {\bf u}^*)=\phi(S^*, (v^*,0), {\bf u}^*)+(0, \hat{z}_2)=(v(\alpha), z^*)+(0, \hat{z}_2)=(v(\alpha), z^*+\hat{z}_2);$$

        \item [(vi)] With time  $S_0:=\frac{2}{\alpha^2}(\bar{z}_1-(z^*+\hat{z}_2))\geq 0$ and control $u\equiv \rho$ we get that
        $$\phi(S_0, (v(\alpha), z^*+\hat{z}_2), \rho)=\left(v(\alpha), S_0\frac{\alpha^2}{2}+(z^*+\hat{z}_2)\right)=(v(\alpha), \bar{z}_1);$$

 \item[(vii)] Now, with control ${\bf u}_1$ and time $S_1$ we get that $\phi(S_1, (v(\alpha), \bar{z}_1), {\bf u}_1)=(0, z_1)$, which gives us the desired trajectory (see Figure \ref{Heisenberg}).
       
\end{itemize}

\end{example}

\pagebreak 

\begin{figure}[h]
	\centering
	\begin{subfigure}{.35\textwidth}
		\centering
		\includegraphics[width=1\linewidth]{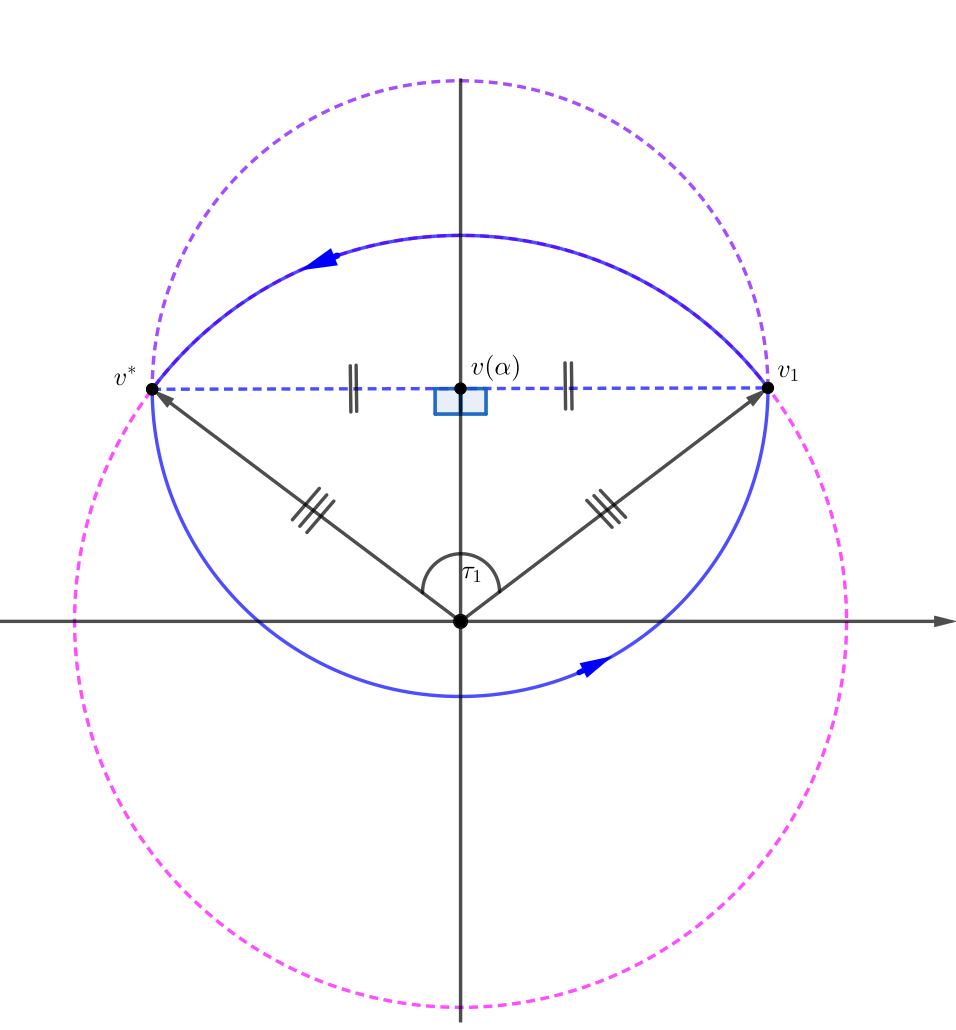}
		\caption{Periodic orbit on the plane.}
		\label{planar}
	\end{subfigure}%
 \hspace{2cm}
	\begin{subfigure}{.35\textwidth}
		\centering
		\includegraphics[width=1\linewidth]{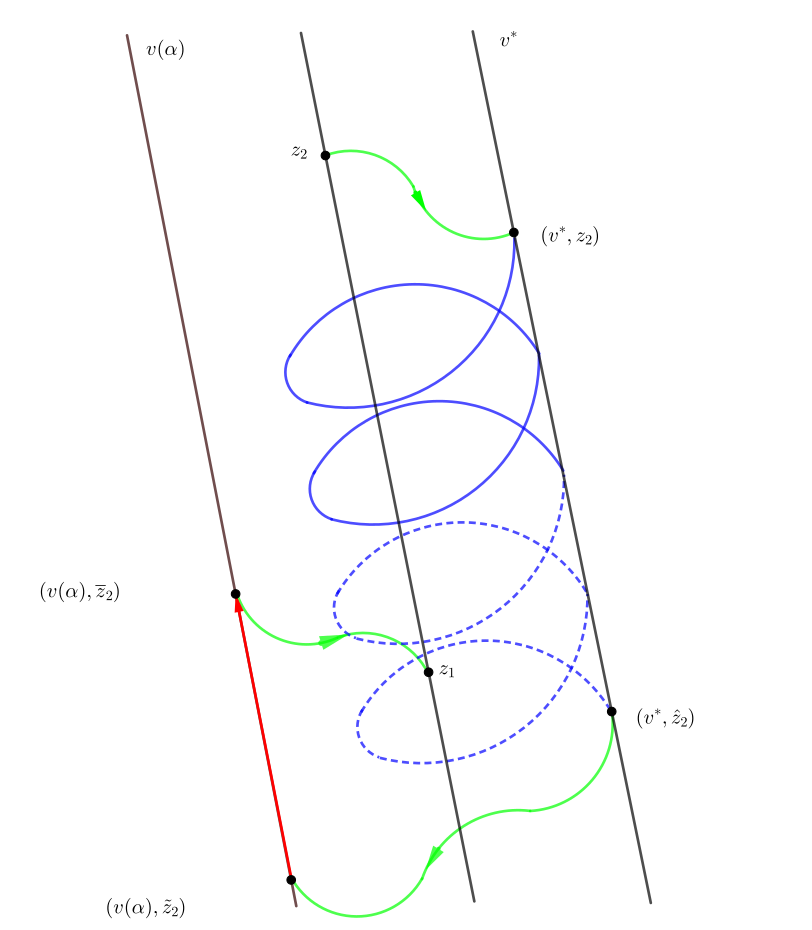}
 		\caption{Trajectory connecting $(0, z_2)$ to $(0, z_1)$.}
		\label{Heisenberg}
	\end{subfigure}
	%\caption{Examples of singular locus}
	%\label{fig:test}
\end{figure}

\bigskip

\begin{remark}
It is important to note that in both examples, the LCS induced on $\R^2$ given by the first two components is controllable. However, the differences in the dynamics of the systems on $\HB$ are quite remarkable, showing that one cannot expect to find a condition assuring the existence of control sets with nonempty interiors for non-regular systems.
\end{remark}

%\subsection{Conclusion and future directions}

 %   From \cite[Theorem 3.3]{VATi} we have that the LARC is the weakest condition to assure transitivity of a LCS. In particular, it is a necessary condition for the existence of control sets with a nonempty interior. Therefore, the results presented here show that the compactness of the set of singularities of the drift of the system is enough to assure the existence of a control set with a nonempty interior. As a consequence, any possible generalization has to find weaker topological conditions for such set. On the other hand, since regularity is directly connected with the singularities of the drift or, in more generality, with the equilibrium points of the systems, a first approach should be the study of the equilibria set of a LCS. 

\subsubsection*{CRediT authorship contribution statement}

Both authors participated in the conceptualization, research, methodology and revision of the article. Adriano Da Silva involved in writing original draft preparation and editing.

\subsubsection*{Declaration of competing interest}

The authors declare that they have no known competing financial interests or personal relationships that could have appeared to influence the work reported in this paper. 

\subsubsection*{Data availability}

 No data was used for the research described in the article. 
 
%\subsubsection*{Funding Information} 
% The first author was supported by Proyecto UTA Mayor Nº 4768-23. The second author was supported by Coordena\c{c}\~ao de Aperfei\c{c}oamento de Pessoal de Nível Superior - Brasil (CAPES) - Finance Code 001.

%\pagebreak

\end{document}